\documentclass[11pt]{article}

\usepackage{amsmath,amssymb,amsthm}

\usepackage{fullpage} % (optional alternative: geometry)

\usepackage{tikz}
\usetikzlibrary{calc,positioning}

\usepackage{enumerate} % keep if you use \begin{enumerate}[(i)] etc.
\usepackage{titlesec}  % keep only if you actually customize headings
\usepackage{xcolor}
\usepackage{todonotes}
\usepackage{relsize}
\usepackage{comment}
\usepackage{csquotes}

\usepackage{url} % or \usepackage{xurl} if you have it (better URL line breaks)
\usepackage[hidelinks]{hyperref}
\usepackage[nameinlink,noabbrev]{cleveref}

\theoremstyle{plain}
\newtheorem{theorem}{Theorem}[section]

\newtheorem{conjecture}[theorem]{Conjecture}

\theoremstyle{definition}
\newtheorem{definition}[theorem]{Definition}
\newtheorem{question}[theorem]{Question}

\theoremstyle{remark}

\newcommand{\eps}{\varepsilon}
\newcommand{\F}{\mathcal{F}}

\newcommand{\Q}{\mathbb{Q}}

\newcommand{\C}{\mathbb{C}}
\newcommand{\R}{\mathbb{R}}
\newcommand{\Z}{\mathbb{Z}}
\newcommand{\N}{\mathbb{N}}

\DeclareMathAccent{\widehat}{\mathord}{largesymbols}{"62}

\title{New Obstacles to Multiple Recurrence}
\author{
Ryan Alweiss
\thanks{Department of Pure Mathematics and Mathematical Statistics and Trinity College, University of Cambridge.
Email: {\tt ra699@cam.ac.uk}. ORCID ID: 0000-0002-7221-4599.
Research supported by an NSF Mathematical Sciences Postdoctoral Fellowship and a Trinity Junior Research Fellowship.  MSC codes 05D10, 11B30, 37A44, 37A46, 37B20.  }\\
}

\begin{document}
\maketitle

\begin{abstract}
We show that there is a set which is not a set of multiple recurrence despite being a set of recurrence for nil-Bohr sets.  This answers Huang, Shao, and Ye's \enquote{higher-order} version of Katznelson's Question on Bohr recurrence and topological recurrence in the negative.  Equivalently, we construct a set $S$ so that there is a finite coloring of $\N$ without monochromatic three-term arithmetic progressions with common differences in $S$, but so that $S$ lacks the usual polynomial obstacles to arithmetic progressions.  This represents a significant contribution to our understanding of when a restricted van der Waerden theorem or a restricted Szemer\'edi Theorem holds, an important problem in additive combinatorics and Ramsey theory.
\end{abstract}

\section{Introduction}\label{section:introduction}

A central and long-standing open problem in the field of topological dynamics is \enquote{Katznelson's Question}, on whether a set of Bohr recurrence is a set of topological recurrence.  This question was raised by a 1968 paper of Veech \cite{veech} and implicitly before that in a 1954 paper of F{\o}lner \cite{folner}.  However, in a classic example of Stigler's law, it is often referred to as \enquote{Katznelson's question}, because Katznelson popularized the question and wrote a beautiful paper \cite{katznelson} connecting it to the chromatic number of abelian Cayley graphs.  In this paper, via a Furstenberg correspondence principle type argument in symbolic dynamics, he shows that Bohr recurrence is equivalent to topological recurrence if and only if every abelian Cayley graph with finite chromatic number has a finite coloring via Bohr sets.

Let us recall the equivalent dynamical and combinatorial definitions of Bohr and topological recurrence, and some basics of topological dynamics.  These appear widely in the literature (see for instance \cite{katznelson} \cite{jg} \cite{gkr}). Recall that $||x||_{\R/\Z}=\min(\{x\},1-\{x\})$ is the distance between $x$ and the nearest integer.

\begin{definition} [Bohr recurrence, combinatorial definition]

A set $S \subset \Z$ is a set of Bohr recurrence if and only if for all nonzero $\alpha_1, \cdots, \alpha_k \in \R$ and all $\eps>0$ there exists some $s \in S$ so that for all $1 \le i \le k$, $||\alpha_i s||_{\R/\Z}<\eps$.

\end{definition}

We will also formally define a Bohr set.

\begin{definition}[Bohr set] A Bohr set is a set $S$ parameterized by a $k$-tuple of reals $(\alpha_1, \cdots, \alpha_k)$ and $\eps$, consisting of the $s$ so that $||\alpha_i s||_{\R/\Z}<\eps$ for all $1 \le i \le k$.  
\end{definition}

The parameter $k$ is called the \emph{dimension} and $\eps$ is the \emph{width}.  Later we will define a complexity of a nil-Bohr set, which will be $\Theta(\max(k,1/\eps))$ for ordinary Bohr sets.

Equivalently a set is a set of Bohr recurrence if it is not contained in the complement of a Bohr set, or equivalently in a finite union of complements of one-dimensional Bohr sets.  Sometimes this is referred to as a \enquote{$Bohr_0$ set} in the literature, with the term Bohr set referring to the more general class of translates of $Bohr_0$ sets.  Here we will only deal with $Bohr_0$ sets, so we will refer to them simply as Bohr sets.

\begin{definition} [Bohr recurrence, dynamical definition]

A set $S \subset \Z$ is a set of Bohr recurrence if and only if for all $k \in \N$ and all rotations $T(x_1, \cdots, x_k)=(x_1+\alpha_1, \cdots, x_k+ \alpha_k)$ on the torus $(\R/\Z)^k$, there exists a point $x \in (\R/\Z)^k$ and a subsequence $S'$ of $n \in S$ so that $\lim_{n \in S'} T^n(x)=x$.

\end{definition}

The two definitions are readily seen to be equivalent.  Since torus rotations are isometries, if $T^nx \to x$ is true for a single point then it is true for all points, so it would be equivalent to change the quantifier on $x$ from \enquote{there exists} to \enquote{for all}.

We only consider $S$ with $0 \notin S$, as if $0 \in S$ then all notions of recurrence are trivial.  In general dealing with $S \subset \N$ is equivalent to dealing with $S \cup -S \subset \Z$, so we will assume for the remainder of this paper that $S \subset \N$.  This is important, because a general topological system is not necessarily invertible, so $T^{-1}$ may not make sense.  

Recall that if $S$ is a subset of a group $G$, the Cayley graph $Cay(G,S)$ on an abelian group $G$ has all elements of $G$ as vertices and $a,b$ in $G$ are adjacent if and only if $a-b \in (S \cup -S)$.  Recall also that the chromatic number of a graph is the minimal number of colors needed to color the vertices so that no two vertices of the same color are adjacent.

\begin{definition} [Topological recurrence, combinatorial definition]

A set $S \subset \N$ is a set of topological recurrence if and only if the Cayley graph $Cay(\Z,S)$ on $\Z$ generated by $S$ has infinite chromatic number.

\end{definition}

Topological recurrence is the most general form of recurrence, applying to all dynamical systems.

\begin{definition} [Topological recurrence, dynamical definition]

A set $S \subset \N$ is a set of topological recurrence if for every compact dynamical system $(X,T)$, where $X$ is a compact metric space and $T$ is a continuous map on $X$, there exists a point $x \in X$ and a subsequence $S'$ of $S$ so that $\lim_{n \in S'} T^nx=x$.

\end{definition}

It was Katznelson \cite{katznelson} who proved the equivalence of these two definitions of topological recurrence, giving the eponymous question its combinatorial form.  One direction is relatively simple.  If a set $S$ is not a set of dynamical topological recurrence, then one can use the dynamical system to come up with a coloring of the Cayley graph. After passing to a minimal subsystem, a Baire category argument
gives some $\eps>0$ so that $d(T^nx,x)>\eps$ for all $x$ and
all sufficiently large $n\in S$. Now cover this subsystem by
finitely many $\eps/3$-balls, and color $m\in\N$ by which ball
$T^m(x)$ ends up in, for a fixed $x$. Refine this coloring by
residue classes modulo a sufficiently large integer to exclude
the finitely many remaining differences. The other direction, using a coloring to construct a dynamical system, follows from an argument of Katznelson that makes use of symbolic dynamics, and bears some resemblance to the Furstenberg correspondence principle.  In \cite{katznelson}, Katznelson also proved that a \emph{lacunary} set $S$, one where consecutive elements have ratios that exceed $1+\delta$ for a constant $\delta>0$, is not a set of Bohr recurrence.  See \cite{ps} for more discussion of this connection to simultaneous Diophantine approximation.  This means that a set of Bohr recurrence cannot be too sparse.

A set of topological recurrence must be a set of Bohr recurrence.  This is readily seen both from the topological definition, as a circle rotation is a special kind of dynamical system, and from the combinatorial version, as if $S$ is a set that is not a set of Bohr recurrence it is easy to construct a finite coloring by taking all of the $n\alpha_i$ modulo $1$, up to an error of $\eps$.  Hence we can state Katznelson's question explicitly as follows.

\begin{question}[Katznelson's question]

If $S$ is a set of Bohr recurrence, then is $S$ necessarily a set of topological recurrence?
	
\end{question}

In \cite{gkr} it is stated that there is \enquote{no consensus among experts as to the expected answer}.

No discussion of the history of Katznelson's question would be complete without also referencing the \enquote{Related Problem}.  Bergelson asked whether there is a set which is a set of topological recurrence but not a set of measurable recurrence, or equivalently whether there is a Cayley graph on the integers with a positive upper density independent set, but infinite chromatic number, and K\v{r}\'{\i}\v{z} \cite{kriz} constructed such a set.  His construction is based on the Boolean Cayley graph over $(\Z/2\Z)^n$ generated by elements of Hamming weight $>n-10\sqrt{n}$.  One can check that elements of Hamming weight $\le n/2-10\sqrt{n}$ form a positive density independent set in this graph.  Furthermore, looking at the elements of Hamming weight $n/2-4\sqrt{n}$, this graph has a copy of the Kneser graph $KG(n, n/2-4\sqrt{n})$ and thus has chromatic number at least $\sqrt{n}$, and goes to infinity with $n$.  K\v{r}\'{\i}\v{z} projects this construction to the integers.  See \cite{jkriz} for an exposition of his example by Griesmer and \cite{jquad} for an extension.

In this paper, we will be concerned with the natural higher-order version of Katznelson's question.  This was first explicitly asked by Huang, Shao, and Ye \cite{hsy1}, is the main subject of their book \cite{hsy}, and was asked again as Question 7 in \cite{gkr}.  Before stating it, we need to define topological multiple recurrence.

\begin{definition}[Topological multiple recurrence, dynamical definition]
A set $S \subset \N$ is a set of topological $d$-recurrence if for every compact dynamical system $(X,T)$, where $X$ is a compact metric space and $T$ is a continuous map on $X$, there exists a point $x \in X$ and a subsequence $S'$ of $S$ so that for all $1 \le d' \le d$, $\lim_{n \in S'} T^{d'n}x=x$.
\end{definition}

In other words, $T^nx, T^{2n}x, \cdots, T^{dn}x$ all go to $x$ along this sequence $S'$ of $n$.  So we have $T^nx \to x, T^{2n}x \to x, \cdots, T^{dn}x \to x$ simultaneously along this sequence $S'$.  For $d=1$, this is the ordinary notion of recurrence.  

\begin{definition}[Topological multiple recurrence, combinatorial definition]
A set $S \subset \N$ is a set of topological $d$-recurrence if for every finite coloring of $\N$, there is a monochromatic $(d+1)$-term arithmetic progression with common difference in $S$.
\end{definition}

Equivalently, the $(d+1)$-hypergraph with edges $(x, x+s, \cdots, x+ds)$ for $s \in S$ has infinite chromatic number.  In other words, restricted van der Waerden with differences in $S$ is true for $(d+1)$-term arithmetic progressions (in particular when $S=\Z$, this is just van der Waerden's theorem \cite{vdw}).  For $d=1$ this is readily seen to coincide with the Cayley graph definition.  As before, this combinatorial definition is equivalent to the dynamical one, see for instance \cite{kra}, Theorem 2.5.  In \cite{hsy}, the notation $\F_{Bir_d}$ is used for the family of sets of topological $d$-recurrence.

For restricted van der Waerden, on top of the usual obstacles in Katznelson's question, there are higher degree obstacles.  If $s$ is far from a multiple of an irrational number like $\pi$, then $x, x+s, x+2s$ cannot all be too close modulo $\pi$.  If $2s^2$ is far from a multiple of $\pi$, then because $x^2-2(x+s)^2+(x+2s)^2=2s^2$, the squares of $x, x+s, x+2s$ cannot all be very close to each other modulo $\pi$ as well, which yields a finite coloring where $x$ is colored by $x^2$ modulo $\pi$ up to a small error.  In general there will be polynomial obstacles up to degree $d$.  So the natural question here is whether there are any other obstacles, with the caveat that one has to also worry about \enquote{generalized polynomials} or \enquote{bracket polynomials} like $s \lfloor s \rfloor$.  These generalized polynomials are formally defined in \cite{hsy}, but we will not define them yet.

In the dynamics literature, these correspond to nilmanifolds.  Multiple recurrence for Bohr sets is equivalent to single recurrence, but there is a higher-order notion of Bohr recurrence called \enquote{nil-Bohr recurrence}.  It was originally defined dynamically in \cite{hk} and is the main object of study in \cite{hsy}.  We state it here for completeness.  For a formal definition of a ($d$-step) nilsystem, one can for instance consult \cite{hsy}, but again it is not so crucial for the present paper.

\begin{definition}[nil-Bohr multiple recurrence, dynamical definition]

A set $S \subset \N$ is a set of $Nil_d$-Bohr recurrence if for every $d$-step nil-system $(X,T)$ there is a point $x \in X$ and a subsequence $S' \subset S$ so that for all $1 \le d' \le d$, $\lim_{n \in S'}T^{d'n}x=x$.
	
\end{definition}

As usual, there is also a combinatorial definition.  Again, we state it here for completeness. 

\begin{definition}[nil-Bohr multiple recurrence, combinatorial definition]

A set $S \subset \N$ is a set of $Nil_d$-Bohr recurrence if for every finite family of generalized polynomials $P_1,\ldots,P_k\in\mathrm{GP}_d$
in the sense of \cite[Section~4.1]{hsy} and every $\eps>0$, there exists $s \in S$ so that $||P_i(s)||_{\R/\Z}<\eps$ for all $1 \le i \le k$.
	
\end{definition}

The set of $n$ so that these generalized polynomials have small $\R/\Z$ is called a \enquote{Nil-Bohr} set, or a \enquote{$Nil_d$-Bohr} set when we want to emphasize the parameter of $d$.  An ordinary Bohr set is just a $Nil_1$-Bohr set.

The equivalence of the definitions follows from
Theorems~A(1) and~B of \cite{hsy}, where the notation $\F_{d,0}$ is used for the family of $Nil_d$-Bohr sets and $\F_{d,0}^{*}$ for the family of sets of $Nil_d$-Bohr recurrence.  Again, we will not formally define \enquote{generalized polynomial}, but it is essentially just a polynomial with floors and ceilings, often also referred to as a \enquote{bracket polynomial} in the literature (see for instance \cite{hsy}, \cite{leibman}).  The \enquote{degree} of a generalized polynomial is just the degree in $s$ ignoring the floors and ceilings, so for instance $s \lfloor s \rfloor$ will have degree $2$.

Huang, Shao, and Ye also prove that one can also equivalently ask for the generalized polynomials $P_i$ to be \enquote{special} generalized polynomials of an even simpler form.  We will define these.  First, following Section 4 of \cite{hsy} we use the notation $[x]:=\lfloor x+1/2 \rfloor$ for the nearest integer function.  

\begin{definition}[Special Generalized Polynomials, after \cite{hsy}]

We define $L$ so that $L(x)=x$ and $L(x_1,\dots,x_\ell)=x_1[L(x_2,\dots,x_\ell)].$

Then the special generalized polynomials of degree at most $d$ are the
expressions
\[
L(n^{j_1}a_1,\dots,n^{j_\ell}a_\ell)
\]
such that $j_i\in\mathbb{N}$ and $a_i\in\mathbb{R}$ for $1\le i\le\ell$,
and $\sum_{i=1}^{\ell} j_i \le d$.

\end{definition}

In particular, we must have $\ell \le d$.

\begin{definition}[nil-Bohr multiple recurrence, simple combinatorial definition]

A set $S \subset \N$ is a set of $Nil_d$-Bohr recurrence if for every finite family of special generalized polynomials $P_1, \cdots, P_k$ with degree at most $d$ and every $\eps$, there exists $s \in S$ so that $||P_i(s)||_{\R/\Z}<\eps$ for all $1 \le i \le k$.
	
\end{definition}

For irrational $\alpha$, we have for instance that the set of $n$ so that $1/4<\{n^2\alpha\}<3/4$ and the set of $n$ so that $1/4<\{ \alpha n[n] \}<3/4$ are not sets of $Nil_2$-Bohr recurrence, even though they are sets of Bohr recurrence.

The equivalence of the simple and ordinary combinatorial definitions of nil-Bohr multiple recurrence is Theorem 4.2.11 of \cite{hsy}.  In our proof, we will instead use a somewhat different form of \enquote{special generalized polynomial} from \cite{leibman}.  Leibman proves in his paper that these have the property that they jointly equi-distribute.

We also note here that for technical reasons, we will also use the equivalent $\ell^1$ notion of a Bohr set, where the norm on $||P_i(S)||_{\R/\Z}$ is instead taken to be the $\ell^1$ norm on $(\R/\Z)^k$ rather than the $\ell^\infty$ norm.  We will speak about $\ell^1$ Bohr sets and standard $\ell^\infty$ Bohr sets.

The natural higher-order analogue of Katznelson's question is whether topological multiple recurrence is equivalent to nil-Bohr multiple recurrence, and specifically whether topological $d$-recurrence is equivalent to $Nil_d$-Bohr recurrence.  This is Question B, the main question asked in \cite{hsy}, which in formulation B-III asks whether $\F_{Bir_d}=\F^{*}_{d,0}$.  In this paper, we answer this question in the negative.

\begin{theorem}~\label{thm:main}	 There is a set $S \subset \N$ which is not a set of topological $2$-recurrence, but is a set of $Nil_2$-Bohr recurrence.
\end{theorem}

We really only need that it is a set of $Nil_2$-Bohr recurrence, but it turns out to be no more difficult to also get this for $d>2$.  The same proof works.

\begin{theorem}~\label{thm:general}There is a set $S \subset \N$ which is not a set of topological $2$-recurrence, but is a set of $Nil_d$-Bohr recurrence for every $d$.
\end{theorem}

Equivalently, we find a set $S$ so that restricted van der Waerden with common difference $s \in S$ is false even for three-term progressions, but $S$ intersects all $Nil_d$-Bohr sets, and so it does not have quadratic obstacles like $\{s: ||s^2\alpha||_{\R/\Z}>\eps\}$ (or higher degree obstacles like $\{s: ||s^3\alpha||_{\R/\Z}>\eps\}$) that would give an \enquote{obvious obstruction}.

Before describing the proofs, we mention a few notable appearances of the restricted van der Waerden theorem in the combinatorics literature.  The restricted versions of these theorems have been an important object of study in additive combinatorics and Ramsey theory for decades.  

\begin{definition}[Brown, Graham, Landman \cite{bgl}]
	A set $S$ is \emph{large} if the restricted van der Waerden theorem is true with common differences in $S$, i.e. any finite coloring of $\N$ has arbitrarily long monochromatic arithmetic progressions with common differences in $S$.
\end{definition}

In other words, a set $S$ is large if it is a set of $d$-topological recurrence for all $d$.  Large sets are the subject of \cite{bgl}.  Brown, Graham, and Landman note for instance that a lacunary set will not be large, though this readily follows from Katznelson's observation \cite{katznelson} that lacunary sets are not sets of Bohr recurrence.  They also note that for instance the set of squares is large, which follows from the polynomial Szemer\'edi Theorem of Bergelson and Leibman \cite{blvdw}.

\begin{definition}[$2$-large, Brown, Graham, Landman \cite{bgl}]
	A set $S$ is \emph{$2$-large} if the restricted van der Waerden theorem is true with common differences in $S$ for two-colorings, i.e. any two-coloring of $\N$ has arbitrarily long monochromatic arithmetic progressions with common differences in $S$.
\end{definition}

Brown, Graham, and Landman ask in their 1999 paper if $2$-large sets are large, but according to the Open Problem Garden \cite{open} this conjecture actually dates back to 1995. Recently we resolved this conjecture \cite{af} based on ideas from this paper.

The closely related restricted Szemer\'edi's Theorem has also been well studied.  For instance, there is work of Bri\"et and Castro-Silva \cite{bc} and Altman \cite{altman} on Szemer\'edi's Theorem with differences from a random set, and Wooley and Ziegler \cite{wz} proved that Szemer\'edi's Theorem is still true when the common difference is restricted to be of the form $p-1$ for $p$ prime.

\section{Proof of~\Cref{thm:general}}

To prove~\Cref{thm:general} we must ultimately construct a set $S \subset \N$.  However, we will first construct a set $S$ in a different setting, and consider an appropriate projection.

Recall that $\ell^1(\N; \R/\Z)$ consists of sequences $(a_1, a_2, a_3, \cdots)$ of elements of $\R/\Z$ so that the $\ell^1$ norm $\sum_{i=1}^{\infty}||a_i||_{\R/\Z}$ is less than $\infty$.  This is an abelian group endowed with the usual addition operation.

The key fact about this group $\ell^1(\N; \R/\Z)$ that we will exploit is that it is not a Hilbert space and so has no true \enquote{quadratic structure} (and similarly, no true cubic structure or quartic structure and so on).  While one can define the square of the $\ell^2$ norm to be $\sum_{i=1}^{\infty}||a_i||^2_{\R/\Z}$ which will be bounded by the usual $\ell^1$ norm $\sum_{i=1}^{\infty}||a_i||_{\R/\Z}$ (and similarly one can define other $\ell^p$ norms), one can check that the parallelogram law $||x+d||^2_{\ell^2}-2||x||^2_{\ell^2}+||x-d||_{\ell^2}^2=2||d||^2_{\ell^2}$ does not hold.  For instance one can set $x=(0, 1/3, 2/3, 0, \cdots)$ and $d=(1/3, 1/3, 1/3, 0, \cdots)$.  Then $||d||^2_2=1/3$ while $||x-d||_{\ell^2}=||x||_{\ell^2}=||x+d||_{\ell^2}$.  However, this \enquote{fake quadratic structure} from this $\ell^2$ norm and from the curvature of the circle will enable us to come up with a coloring that prevents three-term monochromatic progressions.  Over the integers, the $\ell^2$ norm will end up corresponding to something like an infinite sum of squares of fractional parts $\sum \{na_i\}^2$ which is an infinite-dimensional generalized polynomial and does not simplify to a low-dimensional generalized polynomial like $\{n^2 \alpha\}$ where $\alpha=\sum a_i^2$. We now describe the construction.

\subsection{Building the $\ell^1$ example}

Let $\delta_1$ and $\delta_2$ be small positive numbers so that $\delta_1/\delta_2$ is an integer and $\delta_2$ is a suitably small function of $\delta_1$, so for instance $\delta_2<\delta_1^3$.  For example one can pick $\delta_1=10^{-2}$ and $\delta_2=10^{-10}$.  However, for ease of notation we will write $\delta_1$ and $\delta_2$.

\begin{definition}
	
Let $S=S_{\infty}$ be the subset of sequences $(a_1, a_2, \cdots )$ from $\ell^1(\N; \R/\Z)$ so that the following hold:

\begin{enumerate}

  \item There is some $i \in \N$ such that 
  $a_i \in (-\delta_1-2\delta_2,\,-\delta_1+2\delta_2) \subset \mathbb{R}/\mathbb{Z}$.
  
  \item For $j \neq i$ we have $a_j \in (-2\delta_2,2\delta_2) \subset \R/\Z$.

  \item $\delta_1-2\delta_2 < \sum_{j \neq i} ||a_j||_{\R/\Z} < \delta_1+2\delta_2.$
  
\end{enumerate}

\end{definition}

In other words, some $a_i$ is in a small interval around $-\delta_1$ in $\R/\Z$ and the other $a_j$'s are all very close to an integer so that their $\R/\Z$ norms add up to about $\delta_1$.

We also define a finitary version $S_m$ of $S$ supported only on the first $m$ coordinates.  We use the notation $\ell^1(\N_{\le m}; \R/\Z)$ for the space $(\R/\Z)^m$ equipped with the usual norm $||(a_1, \cdots, a_m)||=\sum_{i=1}^{m}||a_i||_{\R/\Z}$.  We will need this later for compactness purposes.

\begin{definition}
Let $\eta_m=2^{-100m}$ and let $S_m$ be the subset of sequences
$(a_1, a_2, \cdots, a_m) \in \ell^1(\N_{\le m}; \R/\Z)$ so that the following hold:
\begin{enumerate}
  \item There is some $i \in \N_{\le m}$ such that
  $a_i \in (-\delta_1-(2-\eta_m)\delta_2,\,-\delta_1+(2-\eta_m)\delta_2) \subset \R/\Z$.

  \item For $j \in \N_{\le m}$ with $j \neq i$ we have
  $a_j \in (-(2-\eta_m)\delta_2,(2-\eta_m)\delta_2) \subset \R/\Z$.

  \item $\delta_1-(2-\eta_m)\delta_2 < \sum_{\substack{j \in \N_{\le m}\\ j \neq i}}
  \|a_j\|_{\R/\Z} < \delta_1+(2-\eta_m)\delta_2.$
\end{enumerate}
\end{definition}

Note that this is the same definition except with the interval bounds $2\delta_2$ replaced by $(2-\eta_m)\delta_2$ for some small constants $\eta_m=2^{-100m} \to 0$.  We must consider such $S_m$ with positive and decreasing $\eta_m$ for compactness purposes.  We may set $\eta_{\infty}=0$ so that $S_{\infty}=S$.

We first check that $S$ meets every Bohr set in this setting.  This will be important later, when we project $S$ down to $\Z$.  Also for $x \in (\R/\Z)^k$ we define $|x|$ to be the maximum of the $\R/\Z$ norms of the coordinates.

\begin{definition}[Bohr set, \(\ell^1(\mathbb N_{\leq m};\mathbb R/\mathbb Z)\)]
For \(s\in\mathbb N\cup\{\infty\}\), write
\[
 T_s=\ell^1(\mathbb N_{\leq s};\mathbb R/\mathbb Z).
\]
We say that
\(B\subset T_m\) is a $\ell^1$-Bohr set of dimension \(k\) and width
\(\epsilon\) if there exist \(r\in\mathbb N\), a continuous homomorphism
\[
 f:T_m\to T_r,
\]
and a subtorus \(Y\subseteq T_r\) of codimension at most \(k\), such that
\[
 B=\{b:d_{\ell^1}(f(b),Y)<\epsilon\},
\]
where
\[
 d_{\ell^1}(x,Y)=\inf_{y\in Y}\|x-y\|_{\ell^1}.
\]
%Equivalently, if \(\pi:T_r\to T_r/Y\) is the quotient map, then
%\[
% B=\{b:\|\pi(f(b))\|_*<\epsilon\}.
%\]
The ambient dimension \(r\) is not part of the dimension of the Bohr set; the
dimension is the codimension bound \(k\).
\end{definition}

We remark that one could similarly define a \(\|\cdot\|_*\)-Bohr set with a different norm like the $\ell^\infty$, but we only need to work with $\ell^1$-Bohr sets in this setting.  In fact, what we actually need is the following finitary statement.

\begin{theorem}~\label{thm:24}
Let \(B\) be a \(\|\cdot\|_{\ell^1}\)-Bohr set with dimension \(k\) and width
\(\epsilon\). Then there exists a function \(g\) so that if
\(m\in\mathbb N\cup\{\infty\}\) and \(m>g(k,\epsilon)\), then
\(S_m\cap B\neq\emptyset\). In fact, we can find such a point whose
\(\ell^1\)-distance from the complement of \(S_m\) is at least
\(\delta_2/100\).
\end{theorem}

\begin{proof}

Up to renaming coordinates, we can express any element of $T_r$ in the form $t+y$, where $t \in (\R/\Z)^k \times \{0\}^{r-k}$ and $y \in Y$, where $Y$ is a subtorus of codimension at most $k$.  Let \(\delta_2 e_j\) be the vector that is \(\delta_2\) in the \(j\)th
coordinate and \(0\) in all other coordinates, and write $f(\delta_2 e_j)=t_j+y_j$. 

Hence, if \(m\) is sufficiently large
as a function of \(k\) and \(\epsilon\), via covering $(\R/\Z)^k$ with small balls we can find a set \(I\subseteq\mathbb N_{\leq m}\) of size \(1+\delta_1/\delta_2\)
with the following property. Calling one element of \(I\) by \(i\), for every
other \(j\in I\) the vectors $t_j$ and $t_i$ satisfy $\|t_j-t_i\|_{\ell^1}\leq \delta_2\epsilon$.

Let \(v\) be the vector with coordinate \(-\delta_1\) in the \(i\)th
coordinate, coordinate \(\delta_2\) in the \(j\)th coordinate for
\(j\in I\), \(j\neq i\), and \(0\) in all other coordinates. Then
\(v\in S_m\) by definition. Since
\[
 v=\sum_{\substack{j\in I\\ j\neq i}}(\delta_2 e_j-\delta_2 e_i),
\]
the triangle inequality gives
\[
 d_{\ell^1}(f(v),Y)
 \leq \sum_{\substack{j\in I\\ j\neq i}}\|t_j-t_i\|_{\ell^1}
 \leq (\delta_1/\delta_2)\delta_2\epsilon
 =\delta_1\epsilon<\epsilon.
\]
Thus \(v\in B\), so \(S_m\cap B\neq\emptyset\).

The point \(v\) may be chosen a fixed positive \(\ell^1\)-distance from the
complement of \(S_m\). Indeed, it has its distinguished coordinate at the
center of the corresponding interval in the definition of \(S_m\), its other
nonzero coordinates at distance \(\delta_2\) from \(0\), and the sum of the
norms of the remaining coordinates exactly equal to \(\delta_1\). Hence
$B_{\ell^1}(v,\delta_2/100)\subset S_m.$
\end{proof}

In the coordinate-torus case, continuous homomorphisms
\(T_\infty\to\mathbb R/\mathbb Z\) are exactly those of the form
\[
 f((a_1,a_2,\dots))=\sum_{i=1}^\infty b_i a_i \quad \text{in }\mathbb R/\mathbb Z,
\]
where \((b_i)_{i\geq1}\in\ell^\infty(\mathbb N;\mathbb Z)\). A continuous
homomorphism to a finite coordinate torus is given by finitely many such
coordinate maps, and the quotient-torus case above is obtained by composing
with the quotient map.

  \subsection{The coloring}

We now describe the coloring on $\ell^1(\N_{\le m}; \R/\Z)$ for $m \in \N \cup \{\infty\}$.  Let $e(x)=e^{2\pi ix}$ be the exponential function from $\R/\Z$ to the unit circle in $\C$.  Note by the mean value theorem that $|e(x)-1| \le 2\pi ||x||_{\R/\Z}$ where the norm on the left-hand side is taken in the complex numbers.  

We will color the sequence $x=(a_1, \cdots, a_m) \in \ell^1(\N_{\le m}; \R/\Z)$  by the value of $f(x)=\sum_{i=1}^{m} \left( e(a_i)-1 \right)$ or in $\C$, taken mod the Gaussian integers $\Z[i]$ up to a small error.  Since $e(x)-1$ is bounded by a constant multiple of the $\R/\Z$ norm, for any $x \in \ell^1(\N; \R/\Z)$, $f(x)$ will be well-defined even when $m=\infty$.  

\begin{definition}~\label{def:color} Partition $\C/\Z[i] \cong (\R/\Z)^2$ into $O(1/\delta^2_2)$ sets of diameter at most $\delta_2$.  Color $(a_1, \cdots, a_m) \in \ell^1(\N_{\le m}; \R/\Z)$ by the set that contains $\sum_{i=1}^{m} (e(a_i)-1)$ taken modulo $\Z[i]$.  
\end{definition}

We now show that our coloring lacks three-term monochromatic progressions with common difference in $S$.

\begin{theorem}
	Let $x,s \in \ell^1(\N_{\le m}; \R/\Z)$ with $s \in S_m$.  Then if $\ell^1(\N_{\le m}; \R/\Z)$ is colored as in~\Cref{def:color}, then $x$, $x+s$, and $x+2s$ will not all receive the same color.\end{theorem}

Before proving this, we recall a basic fact from geometry and trigonometry.  If we consider three equally spaced points on the unit circle, if they cut out an arc which is about a $\delta$ proportion of the circle, then the middle point will have distance about $\delta^2$ from the midpoint of the arc, by the standard Taylor approximation $1-\cos(\delta)=\delta^2/2+O(\delta^3)$.  This distance is called the \enquote{sagitta}.  So the contribution to $f(x)-2f(x+s)+f(x+2s)$ from the big $i$ term so that $a_i \approx \delta_1$ will be about $\delta_1^2$, whereas the other contributions from the $j$ terms will in total at most $\frac{\delta_1}{\delta_2}\delta^2_2 \approx \delta_1 \delta_2$ which is smaller.

\begin{figure}[t]
\centering
\begin{tikzpicture}[scale=3.4, line cap=round, line join=round]

  % ===== parameters =====
  \def\tht{45}   % central angle in degrees (=360*delta)
  \def\phi{60}   % arc centered at angle phi
  % ======================

  % Circle and center
  \draw[thick] (0,0) circle (1);
  \fill (0,0) circle (0.6pt);

  % Endpoints of arc/chord
  \coordinate (P) at ({cos(\phi+\tht/2)},{sin(\phi+\tht/2)});
  \coordinate (Q) at ({cos(\phi-\tht/2)},{sin(\phi-\tht/2)});

  % Arc midpoint A and chord midpoint M
  \coordinate (A) at ({cos(\phi)},{sin(\phi)});
  \coordinate (M) at ({cos(\tht/2)*cos(\phi)},{cos(\tht/2)*sin(\phi)});

  % Radii to show angle
  \draw[thin] (0,0)--(P);
  \draw[thin] (0,0)--(Q);

  % Angle marker + label delta
  \draw[->, thin] (\phi-\tht/2:0.20) arc (\phi-\tht/2:\phi+\tht/2:0.20);
  \node[font=\scriptsize, fill=white, inner sep=1pt]
        at (\phi:0.28) {$\delta$};

  % Arc and chord
  \draw[very thick] (\phi-\tht/2:1) arc (\phi-\tht/2:\phi+\tht/2:1);
  \draw[thin] (P)--(Q);

  % Sagitta (gap) dashed
  \draw[dashed] (M)--(A);

  % Off-to-the-side gap label with arrow
  \coordinate (G) at (\phi+22:1.18);
  \node[font=\scriptsize, fill=white, inner sep=1pt] at (G)
        {gap $\asymp \delta^2$};
  \draw[->, thin] (G) -- ($(M)!0.55!(A)$);

  % Small unlabeled dots on key points (optional but tidy)
  \fill (P) circle (0.5pt);
  \fill (Q) circle (0.5pt);
  \fill (A) circle (0.6pt);
  \fill (M) circle (0.6pt);

\end{tikzpicture}
\caption{Curvature on the unit circle, courtesy of GPT-5: an arc of proportion \(\delta\) (central angle \(2\pi\delta\)) has sagitta
\(1-\cos(\pi\delta)\asymp \delta^2\) between its arc-midpoint and chord-midpoint.}
\label{fig:curvature-gap}
\end{figure}
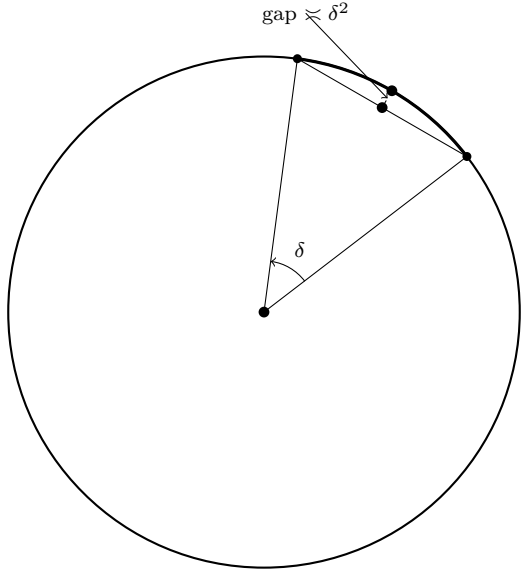

\begin{proof}
	
Let $s=(s_1, \cdots, s_n) \in S \subset \ell^1(\N; \R/\Z)$ and pick $\delta_1$ to be a sufficiently small constant so that $1/\delta_1$ is an integer, and let $\delta_2=\delta_1^5$ be an even smaller constant.  In a slight abuse of notation, we will use $\Theta$ and $O$ for quantities that stay constant as $\delta_1 \to 0$, but we will end up making $\delta_1$ a small enough constant itself.  For any $x=(x_1, \cdots, x_n, \cdots) \in \ell^1(\N; \R/\Z)$, we have that: 

$$f(x)-2f(x+s)+f(x+2s)=e(x_i)-2e(x_i+s_i)+e(x_i+2s_i)+\sum_{j \neq i}\left( e(x_j)-2e(x_j+s_j)+e(x_j+2s_j) \right).$$  

The first term $e(x_i)-2e(x_i+s_i)+e(x_i+2s_i)$ has $\ell^2$ norm $\Theta(\delta_1^2)$ in $\C$ while the other terms have a total $\ell^2$ norm in $\C$ of $$O(\sum_{j \neq i} ||s_j||_{\R/\Z}^2)=O\left((\max_{j \neq i} ||s_j||_{\R/\Z}) (\sum_{j \neq i} ||s_j||_{\R/\Z})\right)=O(
\delta_1 \delta_2)=O(\delta_1^6).$$  So if we choose $\delta_1$ to be a small enough constant, $f(x)-2f(x+s)+f(x+2s)$ has $\ell^2$ norm $\Theta(\delta^2_1)>2 \delta_2$ in $\C$.  This is less than $1/100$, and is far greater than $2 \delta_2$ as well.  This means that $f(x), f(x+s), f(x+2s)$ cannot all be within distance $\delta_2$ of each other in $\C / \Z[i]$, as if they were then $f(x)-2f(x+s)+f(x+2s)=(f(x)-f(x+s))-(f(x+s)-f(x+2s))$ would be within $2\delta_2$ of a Gaussian integer by the triangle inequality.  Hence $x, x+s, x+2s$ cannot all receive the same color under our coloring, and there is no monochromatic $3$-term arithmetic progression with difference $s \in S$. \end{proof}

\subsection{Projecting to $\Z$}

We now project the example $S=S_{\infty}$ from the previous subsection onto a set $S=S_{\N}$ in $\Z$.  Note that in order to avoid a notational clash for $S$, we refer to the sets from the previous section as $S_m$ for $m \in \N \cup \{\infty\}$ and write $S_{\N}=P^{-1}(S_{\infty})\cap\N$ for the pre-image under the map $P$ defined below.

We will pick a sufficiently quickly shrinking sequence of positive real numbers $(\alpha_1,\alpha_2,\ldots)$ such that $1,\alpha_1,\alpha_2,\ldots$ are linearly independent over $\Q$, and consider the projection map $P$ sending $n \to (\{\alpha_1 n\}, \{\alpha_2 n\}, \cdots) \in \ell^1(\N; \R/\Z)$.  We will pick our set of integers $S=S_{\N}$ to be the pre-image of the set $S_{\infty} \subset \ell^1(\N; \R/\Z)$ from the last section under this map.  Since $P$ is a homomorphism, it is immediate to verify that there is a finite coloring of $\Z$ so that there are no monochromatic $3$-APs in $\Z$ with common difference in $S_{\N}$.  Simply color $n \in \Z$ by the color of $P(n) \in \ell^1(\N; \R/\Z)$.  

It remains to check that $S_{\N}$ intersects any nil-Bohr set (and hence every ordinary Bohr set).  We dealt with the Bohr sets in the last section, and to show that higher-order nil-Bohr sets are not an issue we will use standard results about the independence of higher degree generalized polynomials up to easily handled obstructions.  Since $S$ is constructed only from linear Bohr sets $\{n \alpha\}$, it is going to generally be independent of $\{n^2 \alpha\}$ and intersect the set of $n$ so that $||n^2 \alpha||_{\R/\Z}<\eps$.  Of course $\{n^2 \alpha\}$ and $\{(n^2+n)\alpha\}$ will not be linearly independent of $\{n \alpha\}$, and so we may get constantly many more linear Bohr obstructions that can arise in this way.  

Fortunately generalized polynomials can be generated by \enquote{basic generalized polynomials} that do equi-distribute, so the usual Bohr set obstacles are all we need to check.  In particular, by the main result of \cite{leibman}, we may write bounded generalized polynomials as piecewise polynomial functions of these \enquote{basic generalized polynomials}. For instance, a nil-Bohr set generated by $2\sqrt{2}\,n\lfloor\sqrt{2}\,n\rfloor
\equiv -\{n\sqrt{2}\}^2 \pmod{1}$ can instead be expressed in terms of the basic generalized polynomial $\{ n \sqrt{2} \}$.  We remark also that the \enquote{basic generalized polynomials} of \cite{leibman} are slightly different than the \enquote{special generalized polynomials} of \cite{hsy}, and that Leibman uses the abbreviation \enquote{gen-polynomials} for generalized polynomials.  

We now proceed to a rigorous proof. Let $B$ be a nil-Bohr set.
By \cite[Theorems~B and~4.2.11]{hsy}, there are special
generalized polynomials $A_1,\ldots,A_r$ as defined in the
introduction and $\eta>0$ with
\[
\left\{n\in\N:\sum_{i=1}^r\|A_i(n)\|_{\R/\Z}<\eta\right\}\subseteq B.
\]

In particular, $A_i(0)=0$ for every $i$, since these
expressions are formed from monomials $an^j$ with $j\geq1$
using multiplication and nearest-integer parts. The identity $[x]=\lfloor x+1/2\rfloor=x-\{2x\}+\{x\}$
allows us to express nearest-integer parts using ordinary
fractional parts. 

Apply the reduction in the proof of \cite[Theorem~6.1]{leibman},
using a common independent family of ordinary polynomial
generators vanishing at zero. After a sufficiently divisible
rescaling and restriction to $q\N$ to eliminate rational
polynomial residues, we obtain jointly equidistributed
basic polynomials $B_1,\ldots,B_k$ and piecewise polynomial
functions $F_1,\ldots,F_r$ such that
\[
A_i(qn)\equiv F_i(B_1(n),\ldots,B_k(n))\pmod1
\qquad (1\leq i\leq r).
\]

The polynomial pieces of each $F_i$ can be chosen with
integer coefficients, so for some sufficiently small
$\eps>0$ we have

\[
\left\{qn:n\in\N,\
\sum_{j=1}^k\|B_j(n)\|_{\R/\Z}<\eps\right\}
\subseteq B.
\]

Now, pick $m=m(B)$ sufficiently large depending on $B$, and consider the map $(P_m, B)$ sending $n$ to $(\{n\alpha_1\}, \cdots, \{n\alpha_m\}, B_1(n), \cdots, B_k(n))$ where $k=k(B)$. So we define $P_m(n)=(\{n\alpha_1\}, \cdots, \{n\alpha_m\})$ and $B(n)=(B_1(n), \cdots, B_k(n))$.  

A \emph{subtorus} of $(\R/\Z)^{m+k}$ is defined to be the set of points where some set of linear forms with integer coefficients vanish, and the codimension is the number of independent linear forms needed to generate this vanishing set.  For our purposes we need a slight modification of Leibman's result.  Recall that a set $X$ is $\eps$-dense in $Y$ if every point of $Y$ is within $\eps$ of some point in $X$.  In particular, if $X$ is $\eps$-dense in $Y$, it is not necessarily true that $X \subset Y$ or even that $X \cap Y \neq \emptyset$.  Also recall that $X$ is dense in $Y$ if and only if it is $\eps$-dense in $Y$ for all $\eps>0$.  Here and henceforth, by $\eps$-dense we mean in the $\ell^1$ metric.

Throughout the following statements, basic tuples are taken
from a common system as above, with ordinary polynomial
generators vanishing at zero.

\begin{theorem}[Leibman, corollary]~\label{thm:gtlf} Let $B=\{B_1, \cdots, B_k\}$ be a family of jointly equi-distributed basic generalized polynomials and let $\eps>0$.  If $\{1, \alpha_1, \cdots, \alpha_m\}$ are independent  over $\Q$, then the image of the map $n \to (\{n\alpha_1\}, \cdots, \{n\alpha_m\}, B_1(n), \cdots, B_k(n))$ which we denote as $(P_m(n),B(n))$ is $\eps$-dense in some sub-torus of $(\R/\Z)^{m+k}$ of codimension at most $k$. \end{theorem}

\begin{proof}

After adjoining the necessary linear generators and clearing
denominators on a suitable progression $Q\N$,
\cite[Proposition~8.2]{leibman} shows that the joint map
$(P_m(n),B(n))$ is an integer linear image of a jointly
equidistributed tuple, and hence has subtorus closure.
	
The last $k$ coordinates equidistribute over the torus $(\R/\Z)^k$.  Now reveal the coordinates $\{n \alpha_j\}$ one by one.  At each step, the new coordinate $\{n \alpha_j\}$ may or may not satisfy an integer linear relation with the already revealed coordinates.  If there is no such relation, the codimension is unchanged, while if there is such a relation the codimension increases by $1$.  Because the $\alpha_i$ are rationally independent, at the end the projection of the image to the first $m$ coordinates is dense in $(\R/\Z)^m$ by Kronecker.  Hence the projection of the resulting subtorus to the first $m$ coordinates is all of $(\R/\Z)^m$.  Thus the dimension is at least $m$ and so the codimension is at most $k$.  \end{proof}

In~\Cref{thm:gtlf}, it is of course equivalent to restrict $n$ to be bounded by some function $F$ of $(\alpha_1, \cdots, \alpha_m)$, $\eps$, and $B$. 

\begin{theorem}[Leibman, effective corollary]~\label{thm:leibf}
For each \(D\in \N\), there exists a unique minimal function \(F_D\) which
takes as input \(m\)-tuples of irrational numbers
\((\alpha_1,\cdots,\alpha_m)\) with
\(\{1,\alpha_1,\cdots,\alpha_m\}\) linearly independent, \(k\)-tuples
\(B=\{B_1,\cdots,B_k\}\) of jointly equi-distributed basic generalized polynomials with \(k\leq D\),
and \(\eps>0\), with the following property:

The image of the map $n \to (\{n\alpha_1\},\cdots,\{n\alpha_m\},B_1(n),\cdots,B_k(n))$ 
for \(n\in\N\) ranging up to $F_D((\alpha_1,\cdots,\alpha_m),B,\eps)$ is \(\eps\)-dense in some sub-torus of \((\R/\Z)^{m+k}\) of codimension at
most \(D\).
\end{theorem}

\begin{proof}
The existence of such a function is guaranteed by~\Cref{thm:gtlf} whenever
\(D\geq k\), since~\Cref{thm:gtlf} gives codimension at most \(k\).
Taking the pointwise minimum of all such admissible functions yields the
desired minimal \(F_D\).
\end{proof}

We may define a complexity function $C: (B,\eps)=((B_1, \cdots, B_k),\epsilon) \to \N$ where \(B_1,\ldots,B_k\) which are jointly equi-distributed basic generalized polynomials and $\eps>0$ so that $C$ comfortably exceeds all of the parameters of $(B,\eps)$.  For instance we have \(k\leq C\), \(\epsilon\geq 1/C\), $B_i$ have degree at most $C$, there are at most $C$ floor functions and fractional parts in each $B_i$, and so on.  Importantly, for each fixed \(C\), only finitely many formal patterns (e.g. $\{an \lfloor bn \rfloor\}$) occur.
We also require the absolute values of all real coefficients to be at most $C$. For each fixed formal pattern, the coefficients therefore range over a compact cube.

\begin{theorem}[Leibman, refined effective corollary]~\label{thm:refined}
There exist functions \(F'\) and \(D\) with the following property. Let
\((\alpha_1,\ldots,\alpha_m)\) be an \(m\)-tuple of irrational numbers with
\(\{1,\alpha_1,\ldots,\alpha_m\}\) linearly independent over \(\mathbb Q\).
If \(B=(B_1,\ldots,B_k)\) is a \(k\)-tuple of jointly equi-distributed basic
generalized polynomials and \((B,\epsilon)\) has complexity at most \(C\),
then
\[
 F_{D(C)}((\alpha_1,\ldots,\alpha_m),B,\epsilon)
 \leq F'((\alpha_1,\ldots,\alpha_m),C).
\]
Equivalently, the image of
\[
 n\mapsto (\{n\alpha_1\},\ldots,\{n\alpha_m\},B_1(n),\ldots,B_k(n))
\]
for \(n\leq F'((\alpha_1,\ldots,\alpha_m),C)\) is \(\epsilon\)-dense in some
subtorus of \((\mathbb R/\mathbb Z)^{m+k}\) of codimension at most \(D(C)\).
\end{theorem}

In particular, \(D(C)\) is independent of \(m\).

\begin{proof}
We first record the compactness fact that will be used.
Fix $C$, and let $B^{(j)}$ be a sequence of basic
$k$-tuples of complexity at most $C$. After passing to
a subsequence, for every $\rho>0$ there exist $q\geq1$,
$\delta'>0$, and a jointly equidistributed basic tuple
$\overline B$ of length $K'=O_C(1)$ such that, for every
$t\in\N$,
\[
\|\overline B(t)\|_{\ell^1}<\delta'
\quad\Longrightarrow\quad
\limsup_{j\to\infty}\|B^{(j)}(qt)\|_{\ell^1}<\rho.
\]

To prove this, fix the formal patterns and pass to
coefficientwise limits in their compact coefficient cubes.
Write each coordinate as the fractional part of its
corresponding open basic expression. Put
$\langle u\rangle=u-\lfloor u+1/2\rfloor$.
If $u_j\to u$, every accumulation value of $\{u_j\}$
belongs to
\[
\{\langle u\rangle,\langle u\rangle+1\}.
\]
At each internal fractional-part operation, enumerate
these two alternatives recursively. This produces a finite
list of candidate expressions $Q_1,\ldots,Q_L$, with
repetitions, containing every possible accumulation value
of the open expressions at each fixed $n$. Consequently,
\[
\limsup_{j\to\infty}\|B^{(j)}(n)\|_{\ell^1}
\leq
\sum_{\ell=1}^L\|Q_\ell(n)\|_{\R/\Z}.
\]

The number and formal complexities of these candidates
are bounded in terms of $C$. Each candidate retains its
ordinary polynomial factor vanishing at zero and uses
only integer arithmetic outside the ordinary polynomial
inputs. Thus the same near-zero reduction used above
applies jointly to the candidates, giving $q,\delta'$,
and $\overline B$ such that
\[
\|\overline B(t)\|_{\ell^1}<\delta'
\quad\Longrightarrow\quad
\sum_{\ell=1}^L\|Q_\ell(qt)\|_{\R/\Z}<\rho.
\]
The finite polynomial-product span and bounded basic-index
complexity in the proof of \cite[Theorem~6.1]{leibman}
give $K'=O_C(1)$, independently of the limiting coefficients.
This proves the compactness fact.

We now proceed with the proof.  Choose \(D_0(C)\) large enough to dominate the bound \(K'=O_C(1)\) in the
preceding paragraph, and pick some \(D(C)>D_0(C)+C\).  Suppose, for contradiction,
that no such \(F'\) exists for some fixed \((\alpha_1,\ldots,\alpha_m)\).
Then there is a sequence \((B^{(j)},\epsilon_j)\) of jointly equi-distributed
basic generalized polynomial tuples with \(c(B^{(j)},\epsilon_j)\leq C\) such
that
\[
 F_{D(C)}((\alpha_1,\ldots,\alpha_m),B^{(j)},\epsilon_j)\to\infty.
\]
Passing to a subsequence, we may assume \(\epsilon_j\to\epsilon\geq 1/C\) and
that all \(B^{(j)}\) have the same length \(k\).

Apply the compactness fact with $\rho=1/(10C)$, and let
$q,\delta'$, and $\overline B$ be as above.
Since $\{1,q\alpha_1,\ldots,q\alpha_m\}$ is still linearly
independent over $\Q$, \Cref{thm:leibf}, applied to
$\overline B$ and $(q\alpha_1,\ldots,q\alpha_m)$ with
density parameter
$\eta<\min(\delta',1/(10C))$, gives a finite $N$ such that
\[
 \{(P_m(qt),\overline{B}(t)):1\leq t\leq N\}
\]
is \(\eta\)-dense in some subtorus
\(Y\subseteq(\mathbb R/\mathbb Z)^{m+K'}\) of codimension at most \(D_0(C)\).

Let
\[
 Y_0=\{x\in(\mathbb R/\mathbb Z)^m:(x,0)\in Y\}.
\]
Then \(Y_0\) is a subtorus of \((\mathbb R/\mathbb Z)^m\) of codimension at
most \(D_0(C)\).  For every \(x\in Y_0\), the point \((x,0)\) lies in \(Y\),
so by the \(\eta\)-density above there is some \(1\leq t\leq N\) such that
\(P_m(qt)\) is \(\eta\)-close to \(x\) and
\(\|\overline{B}(t)\|<\delta'\). Hence
\[
\limsup_{j\to\infty}\|B^{(j)}(qt)\|_{\ell^1}<1/(10C)
\]
for every $1\leq t\leq N$ satisfying
$\|\overline B(t)\|_{\ell^1}<\delta'$.
Since there are only finitely many such $t$, for all
sufficiently large $j$ we have
\[
\|B^{(j)}(qt)\|_{\ell^1}<1/(3C)
\]
simultaneously for all these witness times. Thus
\[
 \{(P_m(qt),B^{(j)}(qt)):1\leq t\leq N\}
\]
is \(1/(2C)\)-dense in the subtorus
\[
 Y_0\times\{0\}\subseteq(\mathbb R/\mathbb Z)^{m+k}.
\]
This subtorus has codimension at most \(D_0(C)+k\leq D(C)\).  Also
\(\epsilon_j>1/(2C)\) for all sufficiently large \(j\).  Since these are
values of the original map with \(1\leq n\leq qN\), we obtain
\[
 F_{D(C)}((\alpha_1,\ldots,\alpha_m),B^{(j)},\epsilon_j)\leq qN
\]
for all sufficiently large \(j\), contradicting the choice of the sequence.
\end{proof}

Only the number $K'$ of auxiliary basic coordinates needs
to be bounded in terms of $C$. The quantities
$q,\delta',\eta,N$, and the numerical coefficients of
$\overline B$ need not be bounded in terms of $C$;
they only affect the eventual finite bound $qN$.

\begin{theorem}[Leibman, infinitary corollary]~\label{thm:gtli}
There exists a sufficiently quickly shrinking sequence \(\alpha_i\) so that
the following holds. Let \(B=(B_1,\ldots,B_k)\) be a tuple of jointly
equi-distributed basic generalized polynomials, let \(q\in\N\) and \(\epsilon>0\), and put
\(c=c(B,\epsilon)\). If \(m\geq\max(c,q)\), then
the image of the map \[
 n\mapsto
 (\{qn\alpha_1\},\ldots,\{qn\alpha_m\},B_1(n),\ldots,B_k(n))
 =(P_m(qn),B(n))
\]
for \(n\) ranging with $qn$ staying sufficiently smaller than \(1/\alpha_{m+1}\)
is \(\epsilon\)-dense in some subtorus of
\((\mathbb R/\mathbb Z)^{m+k}\) of codimension at most \(D(c)\).
\end{theorem}

\begin{proof}
We may assume that \(F'((\alpha_1,\ldots,\alpha_m),C)\) is nondecreasing in
\(C\). Choose the sequence \(\alpha_i\) inductively, preserving the linear
independence of \(\{1,\alpha_1,\ldots,\alpha_m\}\) over \(\mathbb Q\) for each
\(m\), and decreasing fast enough that \(1/\alpha_{m+1}\) is much larger than

\[
 qF'((q\alpha_1,\ldots,q\alpha_m),m)
\]
for every \(m\) and $q \le m$. Now fix $q$ and \((B,\epsilon)\), and let \(c=c(B,\epsilon)\). For
all \(m\geq \max(c,q)\),~\Cref{thm:refined} gives
\[
 qF_{D(c)}((q\alpha_1,\ldots,q\alpha_m),B,\epsilon)
 \leq qF'((q\alpha_1,\ldots,q\alpha_m),c)
 \leq qF'((q\alpha_1,\ldots,q\alpha_m),m).
\]
By the choice of \(\alpha_{m+1}\), this is sufficiently smaller than
\(1/\alpha_{m+1}\). Hence the required \(\epsilon\)-density occurs with
\(n\) sufficiently smaller than \(1/\alpha_{m+1}\), and with codimension at
most \(D(c)\).
\end{proof}

Note that in the above theorem, if we remove the size bound, and do not require
that \(n\) be small compared to \(1/\alpha_{m+1}\), the image is dense.

\begin{theorem}
If we take an appropriate sequence of \(\alpha_i\) satisfying the conditions
of~\Cref{thm:gtli}, then for any nil-Bohr set \(B\),
\(S_{\mathbb N}\cap B\neq\emptyset\).
\end{theorem}

\begin{proof}
By passing to a basic nil-Bohr subset and shrinking the width if necessary,
we may find inside of our nil-Bohr set some nil-Bohr set of the form
\[
 \{qn:n\in\N,\ \|B_i(n)\|_{\mathbb R/\mathbb Z}<2\theta
       \text{ for }1\leq i\leq k\}\subseteq B.
\]
where \(q\in\N\), \(B_1,\ldots,B_k\) are jointly
equi-distributed basic generalized polynomials, and
\(2\theta<\delta_2/100\). Put
\[
 c=c((B_1,\ldots,B_k),\theta).
\]

Choose \(m \ge q\) sufficiently large so that~\Cref{thm:gtli} applies to
\((B_1,\ldots,B_k)\) with $\ell^1$ density parameter \(\theta\). We also choose \(m\) large enough
so that~\Cref{thm:24} applies to all ordinary $\ell^1$-Bohr
sets with \(\max(k, 1/\eps) = O_c(1)\): each such Bohr set meets \(S_m\) at a point
whose \(\ell^1\)-distance from the complement of \(S_m\) is at least
\(\delta_2/100\).  By~\Cref{thm:gtli}, as \(n\in\mathbb N\) ranges with
\(n\) sufficiently smaller than \(1/\alpha_{m+1}\), the image
\[
 (P_m(qn),B(n))
\]
is \(\theta\)-dense in some subtorus
\(Y\subseteq(\mathbb R/\mathbb Z)^{m+k}\) of codimension \(O_c(1)\).

Consider the set of \(x\in(\mathbb R/\mathbb Z)^m\) such that \((x,0)\) is
within \(\theta\) of \(Y\) in $\ell^1$ distance. This is an ordinary $\ell^1$-Bohr set with
\(\max(k, 1/\eps)=O_c(1)\). Hence we may choose \(x\in S_m\) with
\[
 B_{\ell^1}(x,\delta_2/100)\subset S_m,
\]
and a point \((x_0,y_0)\in Y\) such that
\[
 \|x-x_0\|_{\ell^1}<\theta
 \qquad\text{and}\qquad
 \|y_0\|_{\mathbb R/\mathbb Z} \le \|y_0\|_{\ell^1}< \theta.
\]
By the \(\theta\)-density of the image in \(Y\), there is some \(n\),
sufficiently smaller than \(1/\alpha_{m+1}\), such that
\[
 \|P_m(qn)-x_0\|_{\ell^1}<\theta
 \qquad\text{and}\qquad
 \|B(n)-y_0\|_{\mathbb R/\mathbb Z}<\theta.
\]
Thus \(P_m(qn)\) lies within \(2\theta<\delta_2/100\) of \(x\) in \(\ell^1\),
so \(P_m(qn)\in S_m\). Also
\[
 B(n)\in(-2\theta,2\theta)^k,
\]
so \(qn\) lies in the original nil-Bohr set $B$.

Finally, since \(qn\) is sufficiently smaller than \(1/\alpha_{m+1}\), we may
ensure that
\[
 \sum_{i>m}\|qn\alpha_i\|_{\mathbb R/\mathbb Z}<\eta_m\delta_2.
\]
Since \(P_m(qn)\in S_m\), the definition of \(S_m\) then implies
\(P(qn)\in S_\infty\), and hence \(qn\in S_{\mathbb N}\). Therefore this same
\(qn\) lies in \(S_{\mathbb N}\cap B\), as required.
\end{proof}

A remark about nilmanifolds is in order.  We do not use them here, but our results could be phrased equivalently in those terms.  Section 2 of \cite{leibman} describes the correspondence between equi-distribution of generalized polynomials and the equi-distribution over nilmanifolds from \cite{l3}.  The correspondence between generalized polynomials and nilmanifolds was first proved in \cite{correspondence} and is cited as Lemma 2.3 of \cite{leibman}.  Theorem 2.9 of \cite{gt} is the quantitative version of Theorem 1.9 of \cite{gt}, the main theorem of \cite{l3}.  As written, our construction is not explicit because of the compactness arguments.  Using the quantitative results of Green-Tao \cite{gt}, one can remedy this and make all of our compactness arguments effective.  %We also note that our argument does not rely on the quantitative result of \cite{gt}, and the results of \cite{leibman} are enough to show the existence of a set $S$.

\section{Concluding Remarks and Open Questions}

The most obvious open question here is Katznelson's question.  In view of our disproof of the higher-order version, we conjecture that the answer should be negative.

\begin{conjecture}
There is a set $S \subset \N$ which is a set of Bohr recurrence but not a set of topological recurrence.  

\end{conjecture}

As usual we also state the combinatorial form of this conjecture.

\begin{conjecture}

There is $S \subset \N$ so that the Cayley graph on $\Z$ generated by $S$ has a finite chromatic number, but does not admit a finite coloring via Bohr sets.  

\end{conjecture}

One reason to believe Katznelson's question has a positive answer is that a \enquote{zero density exception} version is known due to \cite{folner}.  In particular, if $A$ is a dense subset of $\Z$, then $A-A$ contains a $1-o(1)$ proportion of a Bohr set.  Katznelson's question is equivalent to the same claim except with $A$ syndetic and without the $o(1)$ proportion, see for instance \cite{gkr}.  A similar \enquote{zero density exception} of the higher-order Katznelson was proved in \cite{hsy}, but we have proved here that higher-order Katznelson has a negative answer.

We generally believe that if Katznelson's question has a positive answer, the proof somehow relies on topology.  Indeed Glasscock, Koutsogiannis, and Richter \cite{gkr} prove the special case of skew products over $(\R/\Z)^2$ by using the intermediate value theorem, and \cite{kriz} relies on the topological lower bound for the chromatic number of the Kneser graph. Indeed the answer for $3$ colors is positive via a topological argument, see \cite{astra}.  

We also note here that the construction in this paper can be expressed as coming from a skew product over $(\R/\Z)^3$.  To see this, consider pairs of complex numbers $(z,w)$ where $z$ is on the unit circle.  We can identify the unit circle in $\C$ with $\R/\Z$ in the obvious way by sending $\alpha \to e(\alpha):=e^{2\pi i \alpha}$.  By taking the second coordinate modulo $\Z[i]$, we can identify these pairs with $(\R/\Z)^3$.

Pick a sufficiently well-approximable $\alpha$, so that $\alpha$ has a sufficiently quickly growing continued fraction expansion.  Now we will consider the map
$T:(z,w)\mapsto (z \cdot e(\alpha),\, w+f(z))$, where
\[
f(z)=\sum_{j \in U}\sin(\pi j \alpha)\,z^j
\]
and $U$ is a sufficiently rapidly growing sequence of integers
chosen so that $\sum_{j\in U}\|j\alpha\|_{\R/\Z}<\infty$.  If
$p=(z\cdot e(-\tfrac{(3N-1)\alpha}{2}), w)$, the second coordinate of $T^{2N}p-T^Np$ is
$\sum_{j \in U} \sin(\pi Nj\alpha)\,z^j$
and the second coordinate of $T^{3N}p-p$ is
$\sum_{j \in U} \sin(3\pi Nj\alpha)\,z^j$.

We will pick $S$ to be the set of $N$ so that for a single $k \in U$, $Nk\alpha$ modulo $1$ is within $\delta_2$ of $-\delta_1$, while for all other $j \in U$ it is within $\delta_2$ of $0$, and so that $\sum_{j \in U,\, j \neq k} ||Nj\alpha||_{\R/\Z}$ is within $\delta_2$ of $\delta_1$.  Now, if $U \ni j \neq k$, then $||Nj\alpha||_{\R/\Z}$ is small and so $\sin(\pi Nj\alpha)$ will be extremely close to $0$, and $\sin(3\pi Nj\alpha)$ will be very close to $3\sin(\pi Nj\alpha)$.  In particular the sum of the magnitudes
\[
\sum_{j \in U,\, j \neq k} \left| \sin(3\pi Nj\alpha)-3\sin(\pi Nj\alpha) \right|
\]
will go to $0$ with $\delta_2$.  However, $\left|\sin(3\pi Nk\alpha)-3\sin(\pi Nk\alpha)\right|$ will be bounded away from $0$ because of the curvature of the circle.  In particular, this means that if we take $T^{3N}p-p-3(T^{2N}p-T^Np)$, it will be far from $0$ when $N \in S$ by inspecting its second coordinate.  Indeed, one can check that this $S$ is the same as the one we construct here.  In particular, since $U$ is very sparse, the values of $Nj\alpha$ modulo $1$ for $j \in U$ are essentially independent.

We mention one important obstacle to extending something like our construction to settle the full Katznelson's question.  Let $G$ be an abelian group.  Then an element of $G^m$ in the generating set with sum of coordinates $0$ can be written in the form $(a_1-a_2, a_2-a_3, \cdots, a_{m-1}-a_m, a_m-a_1)$, so the cyclic shifts $(a_1, a_2, a_3, \cdots, a_m)$ and $(a_2, a_3, \cdots, a_m, a_1)$ are adjacent.  Hence, we cannot simply take a function of each coordinate and add it up the way we considered $(e(x)-1)$ here.  This kind of \enquote{cyclic shift} obstacle destroys many potential counterexamples to Katznelson's question, and any kind of high-dimensional projection must break this sort of symmetry.  Nonetheless, we believe that our construction represents substantial progress towards Katznelson's Question.

Another interesting direction is to consider the finite field setting.  One can also consider analogues of Katznelson's question over positive characteristic.  This was asked by Ben Green over characteristic $2$ (\cite{ben}, Problems 52 and 53) and also by John Griesmer \cite{jg} over more general characteristic $p$.  We can similarly ask the higher-order version of Katznelson's question in characteristic $p$ for $p>d$.  We state two such versions for $3$-term progressions, with the convention that a Cayley $3$-graph on a group $G$ with generating set $S$ will have $3$-edges $\{g, g+s, g+2s\}$ for $g \in G, s \in S$.

\begin{question}~\label{cyclic}
	Let $p_1, p_2, \cdots, p_n, \cdots \ge 3$ be an increasing sequence of primes, and say we have a sequence $G_i$ of Cayley $3$-graphs on $\Z/p_i$ with uniformly bounded chromatic number $C$.  Is there a uniform constant $K$ so that the generating set of $G_i$ is contained in the union of $K$ sets $||P_j(x)/p_i||_{\R/\Z}>1/K$, where $P_j$ are polynomials on $\Z/p_i$ of degree at most $2$ satisfying $P_j(0)=0$?
\end{question}

In other words, can we always find a bounded number $K$ of quadratics $P_j(x)$ so that any $x$ in the generating set is mapped by at least one of the quadratics to a point that is not in the interval $(-p_i/K, p_i/K)$?

\begin{question}~\label{padic}
	Let $p \ge 3$ be a fixed prime, and say we have a sequence of Cayley $3$-graphs $G_i$ on $(\Z/p)^n$ for a sequence of $n=n_i$ with uniformly bounded chromatic number $C$.  Is there a uniform constant $K$ so that the generating set of $G_i$ is contained in the union of $K$ sets $P_j(x) \neq 0$, where $P_j$ are quadratic forms on $(\Z/p)^n$ of degree at most $2$?
\end{question}

A few comments are in order about why our example does not immediately resolve the above questions.  Our example in some sense relies on the infinite bracket polynomial $\sum_i \{na_i\}^2$.  In modulo $1$, $x$ and $x^2$ are independent.  However, in modulo $p$ the value of $n$ determines the value of $n^2$, and $\sum_i (na_i \mod p)^2=n^2(\sum a_i^2) \mod p$.  Our example also relies heavily on the existence of a norm, which is less straightforward in positive characteristic, though one can consider the $p$-adic norm or $||x/p||_{\R/\Z}$.  In general, it is also true that $\Z/p$ and $\R/\Z$ \enquote{look different} from the standpoint of arithmetic progressions of length $3$ or greater.  See \cite{cs} for discussion about this.

One can also ask if any $S$ which is not a set of nil-Bohr recurrence has some kind of \enquote{quadratic structure}.  We generally suspect the answer is no, and in fact that we do not think it is possible to write down a \enquote{nice} set of criteria that determines whether $S$ is a set of multiple recurrence or if $S$ is a set of single recurrence.  There is likely some logical or computability theoretic separation between nil-Bohr recurrence and multiple recurrence, as well as between Bohr recurrence and topological recurrence.  We state a complexity theoretic conjecture in the finitary setting, which assuming $P \neq NP$ implies that the answer to Katznelson's Question is negative in a strong sense.  One could formulate a similar finitary conjecture for the higher-order version.

\begin{conjecture}~\label{conj:ugc}
	There exists an absolute constant $c$ so that for every constant $C$, it is $NP$-hard to distinguish whether an abelian Cayley Graph has chromatic number $\le c$ or $\ge C$. \end{conjecture}

For an abelian Cayley graph, we can determine in polynomial time whether there exists a coloring by Bohr sets using $<C$ colors by simply checking all such colorings, and in \cite{katznelson} it is proved that Katznelson's Question is equivalent to the statement that the number of colors necessary in such a coloring is bounded by a function of the chromatic number.  Hence, if Katznelson's Question has a positive answer, there is a polynomial algorithm for this \enquote{Approximate Graph Coloring} question for abelian Cayley graphs.  For general graphs, such problems are known to be NP-hard conditional on a variant of the Unique Games Conjecture, which gives us reason to believe~\Cref{conj:ugc}.  See \cite{cz} and the references there for more discussion of Approximate Graph Coloring.

Concerning counterexamples with only $2$ colors, recall the $2$-large versus large conjecture of Brown, Graham, and Landman \cite{bgl}.  As mentioned in the introduction, a set $S$ is large if it is a set of $d$-topological recurrence for all $d$, so that restricted van der Waerden with restricted common difference $s \in S$ is true for arbitrarily long monochromatic arithmetic progressions.  Also recall that a set $S$ is $2$-large if any 2-coloring of $\N$ has arbitrarily long monochromatic arithmetic progressions with common difference in $S$, and that Brown, Graham, and Landman asked if $2$-large sets are large.  We constructed the first example of a set that is not large but meets every nil-Bohr set.  In a previous version of this paper, we conjectured that it is necessary to do so to refute the $2$-large versus large conjecture. Recently, this was settled in joint work of the author and Farhangi \cite{af}.

\begin{theorem}[\cite{af}]
	If $S$ is not a set of $Nil_d$-Bohr recurrence for some $d$, then $S$ is not $2$-large.
\end{theorem}

The set $S$ constructed in this paper is not $2$-large, as observed
by GPT-6 Astra. Take $\delta_1=10^{-2}$ and $\delta_2=10^{-10}$
in the construction, and color $n$ by $\lfloor 2F(n)\rfloor \bmod 2$,
where
\[
F(n)=\sum_{j\in\N}\bigl(\cos(2\pi n\alpha_j)-1\bigr).
\]
For $s\in S$, we have
\[
|F(x+s)-F(x)|
\le 2\pi\sum_j\|s\alpha_j\|_{\mathbb{R}/\mathbb{Z}}
<2\pi(2\delta_1+4\delta_2)<\frac12.
\]
Consequently, on a monochromatic progression, all the $F$-values
must lie in a single interval of length $1/2$. However, if $i$ is
the distinguished coordinate of $s$ and $\theta=2\pi n\alpha_i$,
the defining bounds for $S$ give
\[
\begin{aligned}
F(n)-2F(n+25s)+F(n+50s)
  &=-2\sin\theta+\varepsilon_0,\\
F(n+25s)-2F(n+50s)+F(n+75s)
  &=2\cos\theta+\varepsilon_1,
\end{aligned}
\]
where $|\varepsilon_0|,|\varepsilon_1|<10^{-6}$.
At least one of these second differences therefore has magnitude
greater than $1$, whereas three values in an interval of length
$1/2$ have second difference of magnitude less than $1$.
Thus there is no monochromatic $76$-term progression
with common difference in $S$.

Nevertheless, we have the following theorem disproving the $2$-large versus large conjecture, which uses sets of the kind constructed in this paper as a black box.

\begin{theorem}[\cite{af}]
	If $S$ is a set of $Nil_d$-Bohr recurrence for all $d$ and $S$ is not $2$-large, then $S \cup (\N \setminus 3\N)$ is $2$-large but not $3$-large.
\end{theorem}

We can also take the union of $S$ with the odds, as $S \cup (\N \setminus 2\N)$ is also $2$-large but not $3$-large. We do not know if Ben Green's modification \cite{ben} of our example, where the distinguished coordinate can be $1/2$, is $2$-large.

Another potential question we can ask is about the polynomial version of Katznelson's question.  What if we consider the Bergelson-Leibman polynomial van der Waerden Theorem \cite{blvdw} with restricted differences?  We pose one such question.

\begin{question}
	For what sets $S$ does an arbitrary finite coloring of $\N$ always contain a monochromatic $\{x, x+d, x+d^2\}$ with $d \in S$?
\end{question}

Progress on any of these questions would be interesting and would contribute to our general understanding of recurrence. We conclude with an ambitious general conjecture, which by design is not very precisely stated.

\begin{conjecture}[Recurrence Metaconjecture]
For any two natural notions of recurrence $X$ and $Y$, unless the implication from $X$-recurrence
to $Y$-recurrence follows from the known relations between
these notions, there exists a set of $X$-recurrence which
is not a set of $Y$-recurrence.
\end{conjecture}

In other words, we conjecture that the implications known
at the time of writing, together with their logical
consequences, exhaust all implications between these
notions. In particular, this predicts a negative answer
to Katznelson's Question.

\section{Acknowledgments}

We thank GPT-5, Daniel Altman, Zachary Chase, Sohail Farhangi, Yang P. Liu, Julian Sahasrabudhe, and Siming Tu for helpful comments.  Thanks especially to Ben Green for comments greatly improving the projection step. Thanks to GPT-6 Astra for the proof that the set constructed here is not $2$-large, for helpful edits to the proof of the projection step, and for catching a few other minor mistakes.

\end{document}